\documentclass[a4paper,12pt,oneside]{article}
\usepackage[utf8]{inputenc}
\usepackage[english]{babel}
\usepackage{amsmath}
\usepackage{amsthm}
\usepackage{amssymb}
\usepackage{mathtools}
\usepackage{tikz}
\usepackage{verbatim}
\usepackage{fancyvrb}
\usepackage{array}
\usepackage{adjustbox}
\usepackage{changepage}
\usepackage{subcaption}
\usepackage{breqn}
\usepackage{enumerate}
\usepackage{wrapfig}
\usepackage{bbm}
\usepackage{bm}
\usepackage{enumitem}
\usepackage{csquotes}
\usepackage{hyperref}
\usepackage{framed}
\theoremstyle{definition}
\newtheorem{defi}{Definition}[section] 
\newtheorem{thm}[defi]{Theorem} 
\newtheorem{lem}[defi]{Lemma}

\newtheorem{prop}[defi]{Proposition}

\newcommand{\N}{\mathbb{N}}
\newcommand{\Z}{\mathbb{Z}}
\newcommand{\Q}{\mathbb{Q}}
\newcommand{\C}{\mathbb{C}}
\newcommand{\R}{\mathbb{R}}

\newcommand{\p}{\varphi}
\newcommand{\BO}{\mathcal{O}}
\newcommand{\sz}[1]{\left| \vec{#1} \right|}

\title{Rational approximation of Euler's constant using multiple orthogonal polynomials}
\author{Walter Van Assche\footnote{Department of Mathematics, KU Leuven, Celestijnenlaan 200B, Leuven, Belgium. \\
E-mail adresses: \texttt{$\{$thomas.wolfs,walter.vanassche$\}$[at]kuleuven.be}}  \and Thomas Wolfs\footnotemark[1]}
\date{}

\begin{document}
\maketitle

\begin{abstract}
We construct new rational approximants of Euler's constant that improve those of Aptekarev et al. (2007) and Rivoal (2009). The approximants are given in terms of certain (mixed type) multiple orthogonal polynomials associated with the exponential integral. The dual family of multiple orthogonal polynomials leads to new rational approximants of the Gompertz constant that improve those of Aptekarev et al. (2007). Our approach is motivated by the fact that we can reformulate Rivoal's construction in terms of type I multiple Laguerre polynomials of the first kind by making use of the underlying Riemann-Hilbert problem. As a consequence, we can drastically simplify Rivoal's approach, which allows us to study the Diophantine and asymptotic properties of the approximants more easily.
\medbreak

\end{abstract}

\section{Introduction and overview of the results}

Over the years, many have tried (and failed) to prove irrationality of Euler's constant 
$$\gamma = - \int_0^\infty e^{-x} \ln x \, dx. $$
Also for the related Gompertz constant 
$$\delta = \int_0^\infty \frac{e^{-x}}{1+x} dx,$$
a proof for its potential irrationality has not been found yet. However, as discussed in \cite{Apt,PilehroodLag,RivE}, it follows from Siegel’s method for $E$-functions (see, e.g., \cite{ShidlovskiiBook}) that at least one of them is irrational, and even transcendental. A common strategy to (try to) prove irrationality of a given constant is by means of rational approximation: if one can construct good enough rational approximants to the constant, its irrationality follows, see, e.g., the elementary criterion below.

\begin{lem} \label{OQ}
    Let $x\in\R$ and suppose that there exists a sequence $(p_n,q_n)_{n\in\N}$ in $\Z\times\N $ such that $r(n) = - \ln \left| x - p_n/q_n \right| / \ln q_n \in\R$ for all $n$. If $\liminf_{n\to\infty} r(n) > 1 $, then $x$ is irrational.
\end{lem}

The overall quality of the approximants is therefore captured by the exponents $r(n)$ in    
$$ \left| x - \frac{p_n}{q_n} \right| = \frac{1}{q_n^{r(n)}}. $$
In practice, it is typically hard to find rational approximants that prove irrationality of a given constant. It requires a delicate balance of having denominators $q_n$ that don't grow too fast (Diophantine quality) compared to the decay of the absolute error $\left| x - p_n/q_n \right|$ (approximation quality).
\medbreak

One approximation method that has proven to be successful in the past, is known as Hermite-Padé approximation and uses so-called multiple orthogonal polynomials as the denominators in the approximation (see, e.g., \cite{VA}). For example, Apery's celebrated proof of irrationality of $\zeta(3)$ in \cite{Apery} can be interpreted in this framework as explained by Beukers in \cite{Beukers}. Hermite-Padé approximation can also be used to show irrationality of a large class of ${}_pF_q$ hypergeometric series with $p<q$, see \cite{Wolfs}.

Multiple orthogonal polynomials satisfy orthogonality conditions with respect to several weights $(w_1,\dots,w_r)$ and depend on a multi-index $\vec{n}\in\Z_{\geq 0}^r$, of size $\sz{n}=n_1+\dots+n_r$, that determines the way in which these orthogonality conditions are distributed (see \cite[Chapter 23]{Ismail} for an introduction). Suppose that the weights are supported on one interval $\Lambda\subset\R$ and that their moments are all finite. The type I multiple orthogonal polynomials are given by vectors of polynomials $(A_{\vec{n},1},\dots,A_{\vec{n},r})$, with $\deg A_{\vec{n},j}\leq n_j-1$, for which the type I function $F_{\vec{n}} = \sum_{j=1}^r A_{\vec{n},j} w_j $ satisfies the orthogonality conditions
    $$\int_{\Lambda} F_{\vec{n}}(x) x^k dx = 0,\quad k=0,\dots,\sz{n}-2.$$
The type II multiple orthogonal polynomials are polynomials $P_{\vec{n}}$, with $\deg P_{\vec{n}}\leq \sz{n}$, that satisfy the orthogonality conditions
    $$\int_{\Lambda} P_{\vec{n}}(x) x^k w_j(x) dx = 0,\quad k=0,\dots,n_j-1,\quad j=1,\dots,r.$$
The existence of non-zero type I functions and type II polynomials is always guaranteed. One can also consider so-called mixed type multiple orthogonal polynomials, which satisfy a combination of type I and type II orthogonality conditions, see \cite{Sor_mixed} and \cite{DaemsKuijl}. The type~I and type~II settings are connected through the biorthogonality conditions
    $$ \int_{\Lambda} P_{\vec{n}}(x) F_{\vec{m}}(x) dx = \begin{cases}
        1,\quad \sz{n}=\sz{m}-1, \\
        0,\quad \sz{n}\leq \sz{m}-2 \text{ or } \vec{m}\leq\vec{n},
    \end{cases}  $$
and the underlying Riemann-Hilbert problems, see \cite{G-K-VA}. In this sense, both types can be seen as each other's dual.
\medbreak

In the literature, there are two main constructions for Euler's constant that are based on multiple orthogonal polynomials. The first one is due to Aptekarev et al., see \cite{Apt} for a summary of the results. They used type II multiple Jacobi-Laguerre polynomials to construct linear forms $Q_n^{\text{Apt}} \gamma - P_n^{\text{Apt}} \in\N\gamma + \Z$ such that
$$ \left| \gamma - \frac{P_n^{\text{Apt}}}{Q_n^{\text{Apt}}} \right| \asymp \exp(-2\sqrt{2}n^{1/2}),\quad Q_n^{\text{Apt}} \asymp (2n)! n^{-1/4} \exp(\sqrt{2}n^{1/2}),$$
as $n\to\infty$. The notation $x(n)\asymp y(n)$ as $n\to\infty$ means that $\lim_{n\to\infty }x(n)/y(n) \in\R\backslash\{0\}$. In the sense of Lemma \ref{OQ}, the overall quality is determined by
    $$ r_{\text{Apt}}(n) = \frac{\sqrt{2}}{n^{1/2}\ln n} (1+o(1)),\quad n\to\infty. $$
This is rather poor: since $ \lim_{n\to\infty} r_{\text{Apt}}(n) = 0$, the approximants are far from proving irrationality of $\gamma$. The second construction is due to Rivoal, see \cite{Riv}. He uses certain determinants, in which type II multiple Laguerre polynomials of the first kind appear, to produce linear forms $Q_n^{\text{Riv}} \gamma - P_n^{\text{Riv}} \in\N\gamma + \Z$ for which
    $$ \left| \gamma - \frac{P_n^{\text{Riv}}}{Q_n^{\text{Riv}}} \right| = \BO(\exp(-\frac{9}{2}n^{2/3} + \frac{3}{2}n^{1/3})), \quad Q_n^{\text{Riv}} \asymp n!^2 \exp(3n^{2/3} - n^{1/3}), $$
as $n\to\infty$. In the sense of Lemma \ref{OQ}, the overall quality is then determined by
    $$ r_{\text{Riv}}(n) = \frac{9}{2n^{1/3}\ln n} (1+o(1)),\quad n\to\infty, $$
which is a slight improvement over the first construction as the convergence rate to $0$ is slower. In fact, Rivoal's construction is more general as he produces linear forms in $\gamma+\ln x$ for any given $x>0$. It was shown in \cite{Pilehrood} that Rivoal's construction actually leads to an approximation with smaller denominators
    $$ Q_n^{\text{Riv},\ast} \asymp n!\text{lcm}(1,\dots,n) \exp(3n^{2/3} - n^{1/3}), $$
where $\text{lcm}(1,\dots,n)$ denotes the least common multiple of $1,2,\dots,n$. This was also (partly) conjectured by Rivoal, but a proof was only given later in \cite{Pilehrood} by careful analysis of the determinants after exploiting the underlying type II orthogonality conditions. In Section \ref{L_I}, we will explain how Rivoal's construction can be formulated in terms of type~I multiple Laguerre polynomials. Doing so leads to a more straightforward proof of the above results without having to consider any determinants. This interpretation also opens up the road to further improve Rivoal's result. Indeed, it will allow us to construct the approximants described in the theorem below.
\medbreak

\begin{thm} \label{Res_EC}
    Let $x>0$. There exists $(F_{n;1}^{(I)}(x),F_{n;2}^{(I)}(x))_{n\in\N} \subset \Q[x]\times\Q_{>0}[x]$ such that
        $$ \left| (\gamma + \ln x) - \frac{F_{n;1}^{(I)}(x)}{F_{n;2}^{(I)}(x)}\right| = \BO (\exp(-4x^{1/4}n^{3/4} + x^{1/2} n^{1/2} + \frac{3}{8} x^{3/4} n^{1/4})),\quad n\to\infty, $$
    and 
        $$F_{n;2}^{(I)}(x) \asymp n^{-9/8} \exp(4x^{1/4}n^{3/4} - \frac{1}{2} x^{1/2}n^{1/2} - \frac{3}{8} x^{3/4} n^{1/4}),\quad n\to\infty.$$
    Moreover, we have $n! F_{n;2}^{(I)}(x), n!\text{lcm}(1,\dots,n) F_{n;1}^{(I)}(x)\in\Z[x] $.
\end{thm}

We can therefore obtain approximants of $\gamma$ for which the absolute error has a faster sub-exponential decay than in Rivoal's construction. The overall quality is also better, but still far away from proving irrationality of $\gamma$. Indeed, define $P_n^{(I)}=n!\text{lcm}(1,\dots,n) F_{n;1}^{(I)}(1)$ and $Q_n^{(I)} =n!\text{lcm}(1,\dots,n) F_{n;2}^{(I)}(1) $, then, in the sense of Lemma \ref{OQ},
    $$ r^{(I)}(n) = \frac{4}{n^{1/4}\ln n} (1+o(1)),\quad n\to\infty. $$

We refer to Section \ref{L_I} for the precise construction of the approximants. It will be based on the family of mixed type multiple orthogonal polynomials studied in \cite[Section 4.2]{VAWolfs}. The proof of Theorem \ref{Res_EC} will be given in Section \ref{EC_proof}.
\medbreak

In \cite[Section 4.1]{VAWolfs}, one also studies the dual family of mixed type polynomials. We will explain in Section \ref{GC_mot} how this family can be used to construct approximants of 
$$e^{x}E_1(x) = \int_0^\infty \frac{e^{-t}}{x+t} dt,$$
for any given $x>0$. Here $E_{\nu}(x) = \int_1^\infty t^{-\nu} e^{-xt} dt $ denotes the (generalized) exponential integral. In Section \ref{GC_proof}, we will show that these approximants have the following properties.

\begin{thm} \label{Res_GC}
    Let $x>0$. There exists $(F_{n;1}^{(II)}(x),F_{n;2}^{(II)}(x))_{n\in\N} \subset \Q[x]\times\Q_{>0}[x]$ such that
        $$ \left| e^{x}E_1(x) - \frac{F_{n;1}^{(II)}(x)}{F_{n;2}^{(II)}(x)}\right| = \BO(\exp(-4x^{1/4}n^{3/4} - x^{1/2}n^{1/2} + \frac{3}{8} x^{3/4} n^{1/4})),\quad n\to\infty,$$
    and
        $$F_{n;2}^{(II)}(x) \asymp n^{-9/8} \exp(4x^{1/4}n^{3/4} + \frac{1}{2} x^{1/2}n^{1/2} - \frac{3}{8} x^{3/4} n^{1/4}),\quad n\to\infty.$$
    Moreover, we have $n! F_{n;1}^{(I)}(x), n!F_{n;2}^{(I)}(x)\in\Z[x] $.
\end{thm}

The overall quality of the approximants of $\delta=eE_1(1)$ is of the same order as with $\gamma$. Indeed, define $P_n^{(II)}=n!F_{n;1}^{(II)}(1)$ and $Q_n^{(II)} =n! F_{n;2}^{(II)}(1) $, then, in the sense of Lemma \ref{OQ},
    $$ r^{(II)}(n) = \frac{4}{n^{1/4}\ln n} (1+o(1)),\quad n\to\infty. $$
This an improvement compared to the overall quality of the approximants of $\delta$ considered by Aptekarev et al. in \cite[Section 1]{Apt}. They used Laguerre polynomials to construct linear forms $Q_n^{\text{Apt},\ast} \delta - P_n^{\text{Apt},\ast} \in\N\delta + \Z$ such that
    $$ \left| \delta - \frac{P_n^{\text{Apt},\ast}}{Q_n^{\text{Apt},\ast}} \right| \asymp \exp(-4n^{1/2}),\quad Q_n^{\text{Apt},\ast} \asymp n! n^{-1/4} \exp(2n^{1/2}), $$
as $n\to\infty$, and thus, in the sense of Lemma \ref{OQ},
    $$ r_{\text{Apt},\ast}(n) = \frac{4}{n^{1/2}\ln n} (1+o(1)),\quad n\to\infty . $$
These approximants were also studied in \cite[Theorem 2]{PilehroodLag} through a different framework.
\medbreak

It is remarkable that in both of our constructions the approximants arise directly from the mixed type functions themselves. Typically one uses the associated error functions in the Hermite-Padé approximation problem (i.e. their Stieltjes transforms with respect to each type II orthogonality weight), see, e.g., Apéry's proof \cite{Apery}. 
\medbreak

In the last section, we will discuss some ideas on how the approximants might be improved even further.

\section{Approximants for Euler's constant}

In what follows, we will prove Theorem \ref{Res_EC} and provide some motivation for the precise construction of the approximants.

\subsection{Motivation} \label{L_I}

We will first consider a construction based on the type I multiple Laguerre polynomials of the first
kind. Afterwards, we will show that it is essentially equivalent to Rivoal's construction which makes use of the corresponding type II polynomials. Finally, we will propose a modification that leads to the improved approximants in Theorem \ref{Res_EC}.
\medbreak

\textbf{Multiple Laguerre polynomials of the first kind.} The relevant system of weights for the multiple Laguerre polynomials of the first kind is $(w_1(x),w_2(x))=(x^\alpha e^{-x},x^\beta e^{-x})$ on $(0,\infty)$ with $\alpha,\beta>-1$. In what follows, we will denote the type~I functions by $L_{\vec{n}}^{(I\mid \alpha,\beta)}(x) = L_{\vec{n};1}^{(I\mid \alpha,\beta)}(x) x^\alpha + L_{\vec{n};2}^{(I\mid \alpha,\beta)}(x) x^\beta$ and the type II polynomials by $L_{\vec{n}}^{(II\mid \alpha,\beta)}(x)$. We will use the type I functions directly to obtain approximants of $\gamma + \ln x$. This is motivated by the fact that, if $0<\alpha-\beta<1$, the underlying system of weights $(w_1,w_2)$ is a Nikishin system (see \cite{NikiSor}, where this is called an MT-system) on $(0,\infty)$ with
    $$ \frac{w_2(x)}{w_1(x)} = \frac{1}{\Gamma(\beta-\alpha+1)\Gamma(\alpha-\beta)} \int_{-\infty}^0 \frac{(-t)^{\beta-\alpha}}{x-t} dt,$$
see \cite[Eq. 5.12.3]{DLMF}. Indeed, in that case, it can be expected that $L_{\vec{n};1}^{(I\mid \alpha,\beta)}/L_{\vec{n};2}^{(I\mid \alpha,\beta)} \to -w_2/w_1$ as $\sz{n}\to\infty$, e.g., along the diagonal, see \cite[Theorem 1.4]{Lop-Per}.
\medbreak

Explicit expressions for the (normalized) type I polynomials were determined in \cite[Proposition 4.6]{VAWolfs}. It is then straightforward to obtain the integral representation below. Note that here $(s)_k:=\prod_{j=0}^{k-1} (s+j)$ denotes the Pochhammer symbol. 

\begin{prop} \label{L_I_int}
Suppose that $\alpha-\beta\not\in\Z$. Then,
   $$ L_{\vec{n}}^{(I\mid \alpha,\beta)}(x) = (-1)^{\sz{n}+1} \int_\Sigma \frac{1}{(\alpha-t)_{n_1}(\beta-t)_{n_2}} \frac{x^t}{\Gamma(t+1)} \frac{dt}{2\pi i}, $$
   in terms of a counterclockwise contour $\Sigma$ in $\{t\in\C\mid -1< \text{Re}(t) < \min\{\alpha,\beta\}\}$ that encloses $(0,\infty)$ exactly once.
\end{prop}
\begin{proof}
It was proven in \cite[Section 4.2.1]{VAWolfs} that
$$\begin{aligned}
       L_{\vec{n};1}^{(I\mid \alpha,\beta)}(x) &= \frac{(-1)^{\sz{n}+1}}{\Gamma(\alpha+1) (\beta-\alpha)_{n_2} (n_1-1)!} {}_2F_2 \left( \begin{array}{c} -n_1+1,-n_2+\alpha-\beta+1 \\ \alpha+1,\alpha-\beta+1 \end{array}; x \right), \\
    L_{\vec{n};2}^{(I\mid \alpha,\beta)}(x) &= \frac{(-1)^{\sz{n}+1}}{\Gamma(\beta+1) (\alpha-\beta)_{n_1} (n_2-1)!} {}_2F_2 \left( \begin{array}{c} -n_2+1,-n_1-\alpha+\beta+1 \\ \beta+1,-\alpha+\beta+1 \end{array}; x \right).
\end{aligned}$$
The desired representation then follows after an application of the residue theorem.
\end{proof}

We are interested in the limiting case $\beta\to\alpha$. In that case, the second weight becomes  $\lim_{\beta\to\alpha} (x^\beta-x^\alpha)/(\beta-\alpha) = x^\alpha \ln x $
and the type I function is of the form
$$L_{\vec{n}}^{(I\mid \alpha,\alpha)}(x) = (L_{\vec{n};1}^{(I\mid \alpha,\alpha)}(x) + L_{\vec{n};2}^{(I\mid \alpha,\alpha)}(x) \ln x) x^\alpha. $$
The underlying type I polynomials are then determined by the relations
$$L_{\vec{n};1}^{(I\mid \alpha,\alpha)}(x) = \lim_{\beta\to\alpha} [L_{\vec{n};1}^{(I\mid \alpha,\beta)}(x)+L_{\vec{n};2}^{(I\mid \alpha,\beta)}(x)],
\quad  L_{\vec{n};2}^{(I\mid \alpha,\alpha)}(x) = \lim_{\beta\to\alpha} [(\beta-\alpha)L_{\vec{n};2}^{(I\mid \alpha,\beta)}(x)].$$
This procedure is also encoded by the integral representation in Proposition \ref{L_I_int}. In order to obtain less convoluted formulas, we will restrict to the case $\alpha=\beta=0$ and will only consider multi-indices on the diagonal $\vec{n}=(n+1,n+1)$. It is then natural to multiply with the normalization constant $n!^2$.

\begin{prop} \label{L_I_form}
    Denote $L_{n}^{(I)} = n!^2 L_{(n+1,n+1)}^{(I\mid 0,0)}$. Then,
    $$L_{n}^{(I)}(x) = L_{n;1}^{(I)}(x) - L_{n;2}^{(I)}(x) (\gamma + \ln x),$$
    where
    $$L_{n;1}^{(I)}(x) = \sum_{k=0}^{n} \binom{n}{k}^2 \left(  3 H_k - 2 H_{n-k} \right) \frac{x^k}{k!}, \quad    L_{n;2}^{(I)}(x) = \sum_{k=0}^{n} \binom{n}{k}^2 \frac{x^k}{k!},$$
    and $H_l = \sum_{j=1}^l 1/j$ denotes the $l$-th harmonic number.
\end{prop}
\begin{proof}
    The integrand in the contour integral
    $$ L_{n}^{(I)}(x) = - \frac{1}{2\pi i} \int_\Sigma \frac{1}{\Gamma(t+1)} \frac{n!^2}{(-t)_{n+1}^2} x^t dt $$
    has double pole at $t=0,\dots,n$ and thus
    $$L_{n}^{(I)}(x) = - \sum_{k=0}^n \frac{d}{dt}\left[\frac{1}{\Gamma(t+1)} \frac{n!^2}{(-t)_{n+1;k}^2} x^t\right]_{t=k}. $$
   Here we used the notation $(x)_{n;k} = \prod_{j=0,j\neq k}^{n-1} (x+j)$. We therefore have
        $$L_{n}^{(I)}(x) = -  \sum_{k=0}^n \frac{d}{dt}\left[\frac{1}{\Gamma(t+1)} \frac{n!^2}{(-t)_{n+1;k}^2}\right]_{t=k} x^k - \sum_{k=0}^n \frac{1}{\Gamma(k+1)} \frac{n!^2}{(-k)_{n+1;k}^2} x^k \ln x, $$
    and this immediately gives the desired formula for $L_{n;2}^{(I)}(x)$. For $L_{n;1}^{(I)}(x)$, we compute the derivative and obtain
        $$L_{n;1}^{(I)}(x) = \sum_{k=0}^n \frac{\psi(k+1)}{\Gamma(k+1)} \frac{n!^2}{(-k)_{n+1;k}^2} x^k - \sum_{k=0}^n \frac{1}{\Gamma(k+1)} \frac{n!^2}{(-k)_{n+1;k}^2} \sum_{j=0,j\neq k}^n \frac{2}{-k+j} x^k$$
    in terms of the digamma function $\psi(k+1) = \Gamma'(k+1)/\Gamma(k+1) = - \gamma + H_k$. This leads to the stated formula for $L_{n;1}^{(I)}(x)$.
\end{proof}

For $x=1$, the above approximants were already investigated in \cite[Theorem 1.1 ($a=3$)]{PilehroodEC}. Observe that the denominators $L_{n;2}^{(I)}(1) = \sum_{k=0}^{n} \binom{n}{k}^2 \frac{1}{k!}$ are similar to the denominators $ \sum_{k=0}^{n} \binom{n}{k}^2 \binom{n+k}{k}$ that appear in Apéry's proof \cite{Apery} of the irrationality of $\zeta(2)$, but that a factor $(n+k)!/n!$ is missing in the $k$-th term.
\medbreak

The following Diophantine properties are immediate from Proposition \ref{L_I_form}.

\begin{prop}
    $ n! \, \text{lcm}(1,\dots,n) L_{n;1}^{(I)}(x), n! L_{n;2}^{(I)}(x) \in\Z[x]. $
\end{prop}

Using the strategy described in Section \ref{EC_proof}, we can study the asymptotics of the denominators $L_{n;2}^{(I)}(x)$ and errors $L_{n}^{(I)}(x)$. Doing so essentially leads to the improved version of Rivoal's result from \cite{Pilehrood}, but in a more straightforward way without having to consider determinants. We will explain this phenomenon in what follows.
\medbreak

\textbf{Connection to Rivoal's construction.} In \cite{Riv}, Rivoal considers approximants produced by determinants 
$$\begin{vmatrix} L_{n}^{(II)}(z) & E_{n}^{(II)}(z) \\ L_{n+1}^{(II)}(z) & E_{n+1}^{(II)}(z) \end{vmatrix} = Q_n(z) (\gamma + \ln z) - P_n(z) $$
in which he uses type II multiple Laguerre polynomials of the first kind
    $$ L_{n}^{(II)}(x) = L_{(n,n)}^{(II)}(x) / n!^2 , \quad L_{\vec{n}}^{(II)}(x) = (-1)^{n_1+n_2} e^{-x} \frac{d^{n_1}}{dx^{n_1}}[x^{n_1} \frac{d^{n_2}}{dx^{n_2}}[x^{n_2} e^{-x} ]]$$
and a particular combination $E_{n}^{(II)}(z) = E_{n;2}^{(II)}(z) - E_{n;1}^{(II)}(z) \ln z$ of the two errors in the associated Hermite-Padé approximation problem
    $$E_{n;1}^{(II)}(z) = \int_0^\infty \frac{L_n^{(II)}(x)}{z-x} e^{-x} dx,\quad E_{n;2}^{(II)}(z) = \int_0^\infty \frac{L_n^{(II)}(x)}{z-x} e^{-x} \ln x \, dx .$$
Such determinants of objects in the type II setting are connected to the type I functions through Mahler's relation, see \cite[Theorem 23.8.3]{Ismail} or \cite[Section 4]{G-K-VA}, which can be proven via the Riemann-Hilbert problem for the underlying multiple orthogonal polynomials. It was shown in \cite[Theorem 4.1]{G-K-VA} that the Riemann-Hilbert problem yields the following relation

   $$ Y_{\vec{n}}^{(I)} = \begin{pmatrix}
        0 & -I_2 \\
        1 & 0
    \end{pmatrix}
    (Y_{\vec{n}}^{(II)})^{-T}
    \begin{pmatrix}
        0 & 1 \\
        -I_2 & 0
    \end{pmatrix}, $$
where, up to a particular scaling of each row,
    $$ Y_{\vec{n}}^{(I)} = \begin{pmatrix} 
    L_{\vec{n}+\vec{e}_1;1}^{(I)} & L_{\vec{n}+\vec{e}_1;2}^{(I)} & E_{\vec{n}+\vec{e}_1}^{(I)} \\ 
    L_{\vec{n}+\vec{e}_2;1}^{(I)} & L_{\vec{n}+\vec{e}_2;2}^{(I)} & E_{\vec{n}+\vec{e}_2}^{(I)} \\ 
    L_{\vec{n};1}^{(I)} & L_{\vec{n};2}^{(I)} & E_{\vec{n}}^{(I)}
    \end{pmatrix},\quad 
    Y_{\vec{n}}^{(II)} = \begin{pmatrix} L_{\vec{n}}^{(II)} & E_{\vec{n};1}^{(II)} & E_{\vec{n};2}^{(II)} \\
    L_{\vec{n}-\vec{e}_1}^{(II)} & E_{\vec{n}-\vec{e}_1;1}^{(II)} & E_{\vec{n}-\vec{e}_1;2}^{(II)} \\ 
    L_{\vec{n}-\vec{e}_2}^{(II)} & E_{\vec{n}-\vec{e}_2;1}^{(II)} & E_{\vec{n}-\vec{e}_2;2}^{(II)} 
    \end{pmatrix}. $$
Rows associated with multi-indices $\vec{n}+\vec{e}_k$ are scaled such that the entry $L_{\vec{n}+\vec{e}_k;k}^{(I)}(x)$ is monic. Rows associated with multi-indices $\vec{n}-\vec{e}_k$ are scaled such that $L_{\vec{n}-\vec{e}_k}^{(II)}(x)$ satisfies the type I normalization condition with respect to the $k$-th orthogonality weight, i.e.
    $$ \int_0^\infty L_{\vec{n}-\vec{e}_k}^{(II)}(x) x^{n_k-1} w_k(x) dx = 1. $$
This particular scaling ensures that $\det(Y_{\vec{n}}^{(II)})=1$. We may compute the transpose of the inverse of $Y_{\vec{n}}^{(II)}$ via the adjugate formula; since $\det(Y_{\vec{n}}^{(II)})=1$, $(Y_{\vec{n}}^{(II)})^{-T} $ is exactly the cofactor matrix $CF(Y_{\vec{n}}^{(II)})$ of $Y_{\vec{n}}^{(II)}$. We therefore have, up to an appropriate scalar,
\begin{equation} \label{L_I_cof}
    CF(Y_{\vec{n}}^{(II)}) = \begin{vmatrix} 
    E_{\vec{n}}^{(I)} & -L_{\vec{n};1}^{(I)} & -L_{\vec{n};2}^{(I)} \\ 
    -E_{\vec{n}+\vec{e}_1}^{(I)} & L_{\vec{n}+\vec{e}_1;1}^{(I)} & L_{\vec{n}+\vec{e}_1;2}^{(I)} \\ 
    -E_{\vec{n}+\vec{e}_2}^{(I)} & L_{\vec{n}+\vec{e}_2;1}^{(I)} & L_{\vec{n}+\vec{e}_2;2}^{(I)}
    \end{vmatrix},
\end{equation}
and thus,
\begin{equation} \label{L_I_det}
    L_{\vec{n}}^{(I)} = \begin{vmatrix} L_{\vec{n}-\vec{e}_1}^{(II)} & E_{\vec{n}-\vec{e}_1}^{(II)} \\ L_{\vec{n}-\vec{e}_2}^{(II)} & E_{\vec{n}-\vec{e}_2}^{(II)} \end{vmatrix}.
\end{equation}

Observe the similarities with Rivoal's determinant when $\vec{n}=(n,n)$. In practice, it can therefore already be expected that the type I functions will generate approximants that are equivalent to Rivoal's. In order to actually prove this, we require an additional step. It is known that there is a four-term recurrence relation for $L_{\vec{n}}^{(II)}(z)$ using multi-indices on the step-line, see \cite{CoussVA}. We have
    $$z L_{n+1,n}^{(II)}(z) = L_{n+1,n+1}^{(II)}(z) + b_n L_{n+1,n}^{(II)}(z) + c_n L_{n,n}^{(II)}(z) + d_n L_{n,n-1}^{(II)}(z),$$
where
$$b_n=3n+2,\quad c_n=3n^2+3n+1,\quad d_n=n^3,$$
and thus
$$\begin{vmatrix} L_{n,n}^{(II)}(z) & E_{n,n}^{(II)}(z) \\ L_{n+1,n+1}^{(II)}(z) & E_{n+1,n+1}^{(II)}(z) \end{vmatrix} = (z-b_n) \begin{vmatrix} L_{n,n}^{(II)}(z) & E_{n,n}^{(II)}(z) \\ L_{n+1,n}^{(II)}(z) & E_{n+1,n}^{(II)}(z) \end{vmatrix} - d_n \begin{vmatrix} L_{n,n}^{(II)}(z) & E_{n,n}^{(II)}(z) \\ L_{n,n-1}^{(II)}(z) & E_{n,n-1}^{(II)}(z) \end{vmatrix}. $$
If we then use \eqref{L_I_cof} to express $L_{\vec{n}+\vec{e}_2}^{(I)}$ with $\vec{n}=(n+1,n)$ and $L_{\vec{n}+\vec{e}_1}^{(I)}$ with $\vec{n}=(n,n)$ as determinants, similarly as in \eqref{L_I_det}, we get
$$\begin{vmatrix} L_{n,n}^{(II)}(z) & E_{n,n}^{(II)}(z) \\ L_{n+1,n+1}^{(II)}(z) & E_{n+1,n+1}^{(II)}(z) \end{vmatrix} = (b_n-z)  \frac{L_{n+1,n+1}^{(I)}(z)}{\mathcal{L}_{n+1,n;2}} + d_n \frac{L_{n+1,n}^{(I)}(z)}{\mathcal{L}_{n,n;1}}. $$
The appropriate scalars are
$$\mathcal{L}_{n+1,n;2} = \frac{\text{LC}(L_{n+1,n+1;2}^{(I)})}{\int_0^\infty L_{n,n}^{(II)}(x) x^{n} e^{-x} dx},\quad \mathcal{L}_{n,n;1} = \frac{\text{LC}(L_{n+1,n;1}^{(I)})}{\int_0^\infty  L_{n,n-1}^{(II)}(x) x^{n-1} e^{-x} \ln x\, dx}, $$
where $\text{LC}(P)$ denotes the leading coefficient of the input polynomial $P$. It can be shown that $\mathcal{L}_{n+1,n;2} = 1/n!^6$ and $\mathcal{L}_{n,n;1} = n^3/n!^6$ by making use of the contour integral for the type I functions in Proposition \ref{L_I_int} and the Mellin transform of the type II polynomials in \cite[Lemma 4.2]{VAWolfs}. Hence, we find
    $$ (n+1)^2 \begin{vmatrix} L_{n}^{(II)}(z) & E_{n}^{(II)}(z) \\ L_{n+1}^{(II)}(z) & E_{n+1}^{(II)}(z) \end{vmatrix} = (b_n-z) n!^2 L_{n+1,n+1}^{(I)}(z) + n!^2 L_{n+1,n}^{(I)}(z), $$
which provides the precise connection to the type I functions and explains the scaling we proposed for the associated approximants in Proposition \ref{L_I_form}.
\bigbreak


\textbf{Modification of the multiple Laguerre polynomials of the first kind.} We consider a mixed type setting with two systems of weights $(x^{\alpha},x^{\beta})$ and $(e^{-x},E_{\nu+1}(x))$, where $\alpha,\beta,\nu>-1$, as in \cite[Section 4.2]{VAWolfs}. The associated mixed type function 
$$F_{\vec{n},\vec{m}}^{(I\mid \alpha,\beta,\nu)}(x) = F_{\vec{n},\vec{m};1}^{(I\mid \alpha,\beta,\nu)}(x) x^\alpha + F_{\vec{n},\vec{m};2}^{(I\mid \alpha,\beta,\nu)}(x) x^\beta,$$
with $\deg F_{\vec{n},\vec{m};j}^{(I\mid \alpha,\beta,\nu)} = n_j-1 $, satisfies the orthogonality conditions
    $$ \int_0^\infty F_{\vec{n},\vec{m}}^{(I\mid \alpha,\beta,\nu)}(x) x^k e^{-x} dx = 0,\quad k=0,\dots,m_1-1. $$
    $$ \int_0^\infty F_{\vec{n},\vec{m}}^{(I\mid \alpha,\beta,\nu)}(x) x^k E_{\nu+1}(x) dx = 0,\quad k=0,\dots,m_2-1. $$

For appropriate multi-indices, explicit expressions for the mixed type functions were obtained in \cite[Section 4.2]{VAWolfs}. This leads to the following integral representation for the type~I functions. 

\begin{prop} \label{I_int}
Suppose that $\sz{n}=\sz{m}+1$ and $m_1+1\geq m_2$. Assume that $\alpha-\beta\not\in\Z$. Then,
   $$ F_{\vec{n},\vec{m}}^{(I\mid \alpha,\beta,\nu)}(x) =  \frac{(-1)^{\sz{n}+1}}{m_2!} \int_\Sigma  \frac{(t+\nu+1)_{m_2} }{(\alpha-t)_{n_1}(\beta-t)_{n_2}} \frac{x^t}{\Gamma(t+1)} \frac{dt}{2\pi i}, $$
   in terms of a counterclockwise contour $\Sigma$ in $\{t\in\C\mid -1< \text{Re}(t) < \min\{\alpha,\beta\}\}$ that encloses $(0,\infty)$ exactly once.
\end{prop}
\begin{proof}
It was shown in \cite[Theorem 4.8]{VAWolfs} that
$$F_{\vec{n},\vec{m}}^{(I\mid \alpha,\beta,\nu)}(x) = \frac{x^{-\nu}}{m_2!} \frac{d^{m_2}}{dx^{m_2}}\left[x^{m_2+\nu} L_{\vec{n}}^{(I\mid \alpha,\beta)}(x) \right].$$
The desired expression then follows from the integral representation of $L_{\vec{n}}^{(I\mid \alpha,\beta)}(x)$ in Proposition \ref{L_I_int}.
\end{proof}

We are again only interested in the limiting case $\beta\to\alpha$. To ease notation, we will restrict to the case $\alpha=\nu=0$ and will only consider multi-indices of the form $\vec{n} = (n+1,n+1)$ and $\vec{m} = (n+1,n)$. The approximants are then given by $F_{n}^{(I)} = n!^2 F_{(n+1,n+1),(n+1,n)}^{(I\mid 0,0,0)} $. They arise as the following modification of the multiple Laguerre polynomials of the second kind
\begin{equation} \label{I_Rodr}
    F_{n}^{(I)}(x) = \frac{1}{n!} \frac{d^n}{dx^n}[x^n L_n^{(I)}(x)].
\end{equation}

\subsection{Quality} \label{EC_proof}

\textbf{Diophantine quality.} We will study the Diophantine properties of the approximants via representation \eqref{I_Rodr}. In principle, we could also use the integral representation in Proposition \ref{I_int}, similarly as in Proposition \ref{L_I_form}, but this would lead to a sub-optimal Diophantine result.

\begin{prop} \label{I_form}
Denote $F_{n}^{(I)} = n!^2 F_{(n+1,n+1),(n+1,n)}^{(I\mid 0,0,0)} $. Then,
    $$F_{n}^{(I)}(x) = F_{n;1}^{(I)}(x) - F_{n;2}^{(I)}(x) (\gamma + \ln x),$$
where
$$\begin{aligned}
    F_{n;1}^{(I)}(x) &= \sum_{k=0}^{n} \binom{n}{k}^2 \binom{n+k}{k} \left( 3H_k - 2 H_{n-k} \right) \frac{x^k}{k!} 
    + \sum_{k=0}^{n} \binom{n}{k}^2 \frac{x^k}{k!} \sum_{l=1}^{n} \binom{n+k}{n-l} \frac{(-1)^l}{l}, \\
    F_{n;2}^{(I)}(x) &= \sum_{k=0}^{n} \binom{n}{k}^2 \binom{n+k}{k} \frac{x^k}{k!}.
\end{aligned}$$
\end{prop}
\begin{proof}
Note that
    $$ \frac{1}{n!} \frac{d^n}{dx^n}[x^{n+k} (\gamma+\ln x)] = \sum_{l=0}^n \frac{1}{(n-l)!} \frac{d^{n-l}}{dx^{n-l}}[x^{n+k}] \frac{1}{l!} \frac{d^l}{dx^l}[\gamma+\ln x] $$
and thus 
$$ \frac{1}{n!} \frac{d^n}{dx^n}[x^{n+k} (\gamma+\ln x)] = x^k \binom{n+k}{n} (\gamma+\ln x) + x^k \sum_{l=1}^n \binom{n+k}{n-l} \frac{(-1)^{l-1}}{l}. $$
Now use formula \eqref{I_Rodr} and Proposition \ref{L_I_form} to obtain the stated result.
\end{proof}

Observe that the denominators $F_{n;2}^{(I)}(1) = \sum_{k=0}^{n} \binom{n}{k}^2 \binom{n+k}{k} \frac{1}{k!}$ are related to the denominators $\sum_{k=0}^{n} \binom{n}{k}^2 \binom{n+k}{k}^2$ of the approximants in Apéry's proof \cite{Apery} of the irrationality of $\zeta(3)$, but that a factor $(n+k)!/n!$ is missing in the $k$-th term.
\medbreak

The Diophantine properties below then follow immediately from Proposition \ref{I_form}. 

\begin{prop} \label{I_DQ}
    $ n! \, \text{lcm}(1,\dots,n) F_{n;1}^{(I)}(x), n! F_{n;2}^{(I)}(x) \in\Z[x]. $
\end{prop}
\bigbreak

\textbf{Approximation quality.} In order to obtain the asymptotic behavior of the denominators $F_{n;2}^{(I)}(x)$ and errors $F_{n}^{(I)}(x)$ in the approximation, we will proceed as follows. For the errors, we will apply a variation of the saddle point method on the integral representation described in Proposition \ref{I_int}. For the denominators, we will generate an underlying recurrence relation via the \textit{Zeilberger}-command in Maple (see \cite{A=B}). This can be done because the denominators are hypergeometric in nature, see Proposition \ref{I_form}. The asymptotics of solutions of the generated recurrence relation can be analyzed via the Birkhoff-Trjitzinsky theory, see \cite{BT} or \cite{WimpZeil}. In general, this theory implies that an $(r+1)$-term recurrence relation, with coefficients $C_{n,j}$ that admit a Poincaré type expansion
    $$ C_{n,j} \asymp n^{\kappa_j/\omega} \sum_{k=0}^\infty c_{j,k} n^{-k/\omega},\quad n\to\infty ,$$
($\kappa_j\in\Z$, $\omega\in\Z_{\geq 1}$, $c_{j,k}\in\C$) has $r$ solutions $S_{n,i}$ that form a basis over $\C$ for the space of solutions of the recurrence relation and have an asymptotic expansion of the form
    $$S_{n,i} \asymp n^{\alpha n + \beta} \exp\left(\sum_{j=0}^{\rho-1} \mu_j n^{(\rho-j)/\rho}\right) \sum_{k=0}^K (\ln n)^k n^{r_{K-k}/\rho} \sum_{l=0}^\infty b_{k,l} n^{-l/\rho}, \quad n\to\infty, $$
($\alpha,\beta,\mu_j,b_{k,l}\in\C$, $\rho,\alpha\rho,r_j\in\Z$, $\rho\geq 1$, $r_0=0$) with $\rho\in\omega \Z_{\geq 1}$.
\medbreak

By applying the \textit{Zeilberger}-command in Maple to $F_{n;2}^{(I)}(x)$, we can find a recurrence relation that leads to the following result.

\begin{prop}
    The denominators $F_{n;2}^{(I)}(x)$ arise as solutions of a five-term recurrence relation $\sum_{k=0}^4 c_{n,k}(x) S_{n+k}(x)=0$ in which $c_{n,k}(x)\in\Z[x,n]$.
    For $x=1$, the coefficients are given by
    $$c_{n,0}(1) = -(n + 1)^3(n + 2)(729n^4 + 7866n^3 + 31485n^2 + 55420n + 36204), $$
    $$\begin{aligned}
        c_{n,1}(1) = (n + 2)&(2916n^7 + 56979n^6 + 457068n^5 + 1963246n^4 \\
        &+ 4896596n^3 + 7111851n^2 + 5581516n + 1829348),
    \end{aligned} $$
    $$\begin{aligned}
        c_{n,2}(1) = &-4374n^8 - 41364n^7 - 17307n^6 + 1312308n^5 + 7426532n^4 \\
        &+ 19463372n^3 + 27733489n^2 + 20769508n + 6413772
    \end{aligned}$$
    $$\begin{aligned}
        c_{n,3}(1) = (n + 3)&(2916n^7 + 56250n^6 + 443064n^5 + 1848631n^4 \\
    &+ 4411534n^3 + 6018161n^2 + 4343036n + 1278012),
    \end{aligned}$$
    $$c_{n,4}(1) = -(n + 3)(n + 4)^3(729n^4 + 4950n^3 + 12261n^2 + 13132n + 5132).$$  
\end{prop}

Even though it is not required in what follows, it is worthwhile to note that it can be shown that the three other linear independent solutions of the recurrence relation are $F_{n;1}^{(I)}(x)$ and the Stieltjes transforms of $F_{n}^{(I)}$ with respect to each type II orthogonality measure, i.e.
    $$E_{n;1}^{(I)}(x) = \int_0^\infty \frac{F_{n}^{(I)}(t)}{x-t} e^{-t} dt ,\quad E_{n;2}^{(I)}(x) = \int_0^\infty \frac{F_{n}^{(I)}(t)}{x-t} E_{1}(t) dt.$$
Hence this recurrence relation actually governs the whole underlying Hermite-Padé approximation problem. We may prove this in the following way. We can show that $F_{n}^{(I)}(x)$ is a solution of the recurrence relation via the contour integral representation in Proposition \ref{I_int}, hence the linear combination $F_{n;1}^{(I)}(x)=F_{n}^{(I)}(x)-F_{n;2}^{(I)}(x)(\gamma+\ln x)$ must be as well. By exploiting the orthogonality conditions on $F_{n}^{(I)}(x)$, similarly as in \cite[Proposition 7]{Riv}, we can then prove that $E_{n;1}^{(I)}(x)$ and $E_{n;2}^{(I)}(x)$ are solutions as well. In order to prove linear independence, we have to show that
    $$ \Delta_n(x) = \begin{vmatrix} F_{n;1}^{(I)}(x) & F_{n;2}^{(I)}(x) & E_{n;1}^{(I)}(x) & E_{n;2}^{(I)}(x) \\ 
    F_{n+1;1}^{(I)}(x) & F_{n+1;2}^{(I)}(x) & E_{n+1;1}^{(I)}(x) & E_{n+1;2}^{(I)}(x) \\  
    F_{n+2;1}^{(I)}(x) & F_{n+2;2}^{(I)}(x) & E_{n+2;1}^{(I)}(x) & E_{n+2;2}^{(I)}(x) \\ 
    F_{n+3;1}^{(I)}(x) & F_{n+3;2}^{(I)}(x) & E_{n+3;1}^{(I)}(x) & E_{n+3;2}^{(I)}(x) \\  \end{vmatrix} \neq 0.$$ 
By making use of the recurrence relation $\sum_{k=0}^4 c_{n-1,k}(x)/c_{n-1,4}(x)  S_{n+k-1}(x)=0$, we can reduce the index of the last row to $n-1$, which gives 
    $$\Delta_n(x) = - \frac{c_{n-1,0}(x)}{c_{n-1,4}(x)} \Delta_{n-1}(x).$$
Repeating this procedure $n-1$ more times yields
    $$ \Delta_n(x) = (-1)^n \Delta_{0}(x) \prod_{j=1}^{n} \frac{c_{n-j,0}(x)}{c_{n-j,4}(x)} . $$    
The remaining determinant $\Delta_{0}(x)$ can then be computed explicitly. In order to this, we may replace the columns with Stieltjes transforms $E_{n;1}^{(I)}(x)$ and $E_{n;2}^{(I)}(x)$ by columns with the corresponding numerator polynomials from the associated Hermite-Padé problem.
\medbreak

The limiting characteristic polynomial $-729(\lambda-1)^4$ of the generated recurrence relation in $x=1$ has a single root at $\lambda=1$, hence the $n$-th roots of solutions of the recurrence relation converge to one or become zero for large $n$, see the extension of Poincaré's lemma in \cite{Pituk}. In order to obtain more precise asymptotics, we have to switch to the more powerful Birkhoff-Trjitzinsky theory \cite{BT}.

\begin{prop} \label{I_AQ_BT}
    There are four solutions $\{S_n^{(I)}(\omega) \mid \omega^4=x\}$ that form a basis over $\C$ for the space of solutions of the recurrence relation for $F_{n;2}^{(II)}(x)$ with asymptotic behavior
    $$S_{n}^{(I)}(\omega) = n^{-9/8} \exp\left(4\omega n^{3/4} - \frac{1}{2}\omega^2 n^{1/2} - \frac{3}{8}\omega^3 n^{1/4} \right) \left(1+\BO(n^{-1/4})\right),\quad n\to\infty.$$
\end{prop}
\begin{proof}
    We will use the strategy proposed in \cite{WimpZeil} (and the accompanying Maple-package \textit{AsyRec} \cite{Zeil}). There one suggests to first try to obtain an asymptotic expansion of the form 
    $$n^{\beta} \exp\left(\sum_{j=1}^{\rho-1} \mu_{j} n^{(\rho-j)/\rho}\right) \left(1+\BO(n^{-1/\rho})\right),\quad n\to\infty $$
    with no logarithmic terms. Note that since we have already established that the $n$-th roots of solutions of the recurrence relation converge to zero, we could assume that $\alpha=\mu_0=0$. We can then try to fit such asymptotics to the recurrence relation. To this end, we consider the normalized relation $1 + \sum_{k=1}^4 c_{n,k}(z)/c_{n,0}(x) S_{n+k}(x)/S_{n}(x) = 0 $, use that 
        $$\frac{S_{n+k}(z)}{S_{n}(x)} = (1+k/n)^{\beta} \exp\left(\sum_{j=1}^{\rho-1} \mu_j n^{(\rho-j)/\rho}((1+k/n)^{(\rho-j)/\rho}-1)\right) \left(1+\BO(n^{-1/\rho})\right) $$
    and expand the powers of $1+k/n$ asymptotically up to an appropriate order. Doing so leads to certain equations for the $\beta$ and $\mu_j$. If we set $\rho=4$, the desired result follows after some computations in Maple.
\end{proof}

The possible exponents for $x=1$ are summarized in the table below.

\begin{center}
\begin{tabular}{c||c|c|c}
    $S_n^{(I)}(\omega)$ & $n^{3/4}$ & $n^{1/2}$ & $n^{1/4}$ \\ \hline
    $1$ & $4$ & $-1/2$ & $-3/8$ \\
    $i$ & $(4i)$ & $1/2$ & $(3/8i)$ \\
    $-i$ & $(-4i)$ & $1/2$ & $(-3/8i)$ \\
    $-1$ & $-4$ & $-1/2$ & $3/8$ \\
\end{tabular}
\end{center}

By providing a suitable estimate for the denominator, we can associate a specific asymptotic behavior to it.

\begin{prop} \label{I_AQ_denom}
Suppose that $x>0$. Then $F_{n,2}^{(I)}(x) \asymp S_{n}^{(I)}(\sqrt[4]{x})$ as $n\to\infty$.
\end{prop}
\begin{proof}
Since $x>0$, we can use the estimate
    $$F_{n,2}^{(I)}(x) \geq \sum_{k=0}^{n} \binom{n}{k} \frac{x^k}{k!} = L_{n}(-x) $$
in terms of the classical Laguerre orthogonal polynomial 
$$L_n(x) = \frac{1}{n!} \frac{d^n}{dx^n}[x^n e^{-x}],$$
see, e.g., \cite[Chapter 5]{Szego}. According to Perron's formula \cite[Theorem 8.22.3]{Szego}, we have
$$L_{n}(-x) \asymp n^{-1/4} \exp\left(2\sqrt{nx} \right),\quad n\to\infty.$$
Hence the coefficient of $S_{n}^{(I)}(\sqrt[4]{x})$ in the expansion of $F_{n,2}^{(I)}(x)$ with respect to the basis $\{S_n^{(I)}(\omega) \mid \omega^4=x\}$ can not be zero and we must have the above.
\end{proof}

The asymptotics of the error can be obtained by applying a variation of the saddle point method on the integral representation in Proposition \ref{I_int}. As expected, these asymptotics are also determined by the solutions $S_n^{(I)}(\omega)$.

\begin{prop} \label{I_AQ_error}
    Suppose that $x>0$. Then,
    $$F_{n}^{(I)}(x) \in\text{span}\{\text{Re}(S_{n}^{(I)}(i\sqrt[4]{x})),\text{Im}(S_{n}^{(I)}(i\sqrt[4]{x})),S_n^{(I)}(-\sqrt[4]{x})\}.$$     
\end{prop}
\begin{proof}
By making use of the relation $(-t)_{n+1} = (-1)^{n+1} (t-n)_{n+1}$, we may write the contour integral in Proposition \ref{I_int} as 
    $$F_n^{(I)}(x) = n! \int_{\Sigma} f_n(t) x^t \frac{dt}{2\pi i},\quad f_n(t) = \frac{\Gamma(t+n+1)\Gamma(t-n)^2}{\Gamma(t+1)^4}.$$
Afterwards, we perform a change of variables of the form $t\mapsto n^{\alpha} s$ for some $0<\alpha<1$ that will be chosen later. We will now approximate the gamma functions in the integrand via Stirling's formula
    $$\Gamma(s+a) \sim (2\pi)^{\frac{1}{2}} e^{-s} s^{s+a-\frac{1}{2}},\quad |s|\to\infty,\ \arg(s) < \pi. $$
The notation $f(t)\sim g(t)$ as $t\to\infty$ means that $\lim_{t\to\infty }f(t)/g(t) = 1$. We find that
$$f_n(n^{\alpha}s) \sim \frac{1}{(2\pi)^{\frac{1}{2}}} e^{n^{\alpha}s+n} \frac{(n^{\alpha}s+n)^{n^{\alpha}s+n+\frac{1}{2}} (n^{\alpha}s-n)^{2n^{\alpha}s-2n-1}}{(n^{\alpha}s)^{4n^{\alpha}s+2}},\quad n\to\infty. $$
By collecting some factors of powers of $n$ in the ratio, we may rewrite it as 
$$ n^{(3-4\alpha)n^{\alpha}s - n - \frac{1}{2} -2\alpha} \frac{(1+n^{\alpha-1}s)^{n^{\alpha}s+n+\frac{1}{2}} (n^{\alpha-1}s-1)^{2n^{\alpha}s-2n-1}}{s^{4n^{\alpha}s+2}}. $$
Since in what follows, we would like to apply the saddle point method, it makes sense to take $\alpha=3/4$ so that the $n^\alpha \ln n$ term in the exponent disappears. Doing so yields,
$$f_n(n^{3/4}s) \sim \frac{e^{n}}{(2\pi)^{\frac{1}{2}} n^{n+2}} e^{n^{3/4}s} \frac{(1+n^{-1/4}s)^{n^{3/4}s+n+\frac{1}{2}} (n^{-1/4}s-1)^{2n^{3/4}s-2n-1}}{s^{4n^{3/4}s+2}}, $$
which we may simplify to
$$f_n(n^{3/4}s;x) \sim - \frac{n^{-3/2}}{n!} e^{n^{3/4}s} \frac{(1+n^{-1/4}s)^{n^{3/4}s+n} (1-n^{-1/4}s)^{2n^{3/4}s-2n}}{s^{4n^{3/4}s+2}}, $$
by applying Stirling's formula to $n!$ and by making use of the fact that $(1+n^{-1/4}s)^{\frac{1}{2}}\sim 1$ and $(n^{-1/4}s-1)^{-1}\sim -1$ as $n\to\infty$. Therefore,
$$ n! n^{3/4} f_n(n^{3/4}s) x^{n^{3/4}s} \sim - s^{-2} n^{-3/4} e^{n^{3/4}\p_n(s)},  $$
where
    $$\p_n(s) = s(1-4\ln s+\ln x) + (s+n^{1/4})\ln(1+n^{-1/4}s) + 2 (s-n^{1/4}) \ln(1-n^{-1/4}s).$$  
We can use Maple to look for approximate saddle points $s^\ast_n = a+bn^{-1/4}+cn^{-1/2}+O(n^{-3/4})$ such that $\p_n'(s^\ast_n)=O(n^{-3/4})$ as $n\to\infty$ (to this end, we will expand the logarithms). If we would only consider leading terms, we would have to look for the actual saddle points of $\p(s)=s(4-4\ln s+\ln x)$, which are given by the solutions of $s^4=x$. In general, we get a system of equations that we can solve for the values of the parameters. We then find four of such approximate saddle points
    $$s^\ast_n(\omega) = \omega - \frac{1}{4} \omega^2 n^{-1/4} - \frac{9}{32} \omega^3 n^{-1/2} + O(n^{-3/4}),\quad \omega^4=x, $$
and we have
$$\p_n(s^\ast_n(\omega)) = 4\omega -\frac{1}{2} \omega^2 n^{-1/4} - \frac{3}{8} \omega^3 n^{-1/2} + O(n^{-3/4}). $$
By applying the (modified) saddle point method, we can then show that $F_n^{(I)}(x)$ is asymptotic to an appropriate linear combination of $n^{-9/8}\exp\left(n^{3/4}\p_n(s^\ast_n(\omega))\right)$ for $\omega^4=x$. Observe that these possible asymptotics are exactly those found in Theorem \ref{I_AQ_BT} via the Birkhoff-Trjitzinsky theory. The specific asymptotics that appear in the combination here depend on the approximate saddle points that the curve of steepest descent, which the contour $n^{-3/4}\Sigma$ needs to be deformed to, passes through. Note that the contour $n^{-3/4}\Sigma$ can't be deformed to a curve that passes through the approximate saddle point $s_n^\ast(x^{1/4})$ on the positive real line because then it would need to cross the strip $[0,n^{1/4}]$ inside $n^{-3/4}\Sigma$ where $f_n(t)$ has its poles. Hence we must have the stated result. 
\end{proof}

\section{Approximants for the Gompertz constant}

In this section, we will prove Theorem \ref{Res_GC} and provide some motivation for the precise construction of the approximants.

\subsection{Motivation} \label{GC_mot}

We will consider the mixed type functions that are dual to the ones used in the previous section. The relevant systems of weights are $(e^{-x},E_{\nu+1}(x))$ and $(x^{\alpha},x^{\beta})$, where $\alpha,\beta,\nu>-1$, as in \cite[Section 4.1]{VAWolfs}. The associated mixed type function 
$$F_{\vec{n},\vec{m}}^{(II\mid \alpha,\beta,\nu)}(x) = F_{\vec{n},\vec{m};1}^{(II\mid \alpha,\beta,\nu)}(x) e^{-x} + F_{\vec{n},\vec{m};2}^{(II\mid \alpha,\beta,\nu)}(x) E_{\nu+1}(x),$$
with $\deg F_{\vec{n},\vec{m};j}^{(II\mid \alpha,\beta,\nu)} = n_j-1 $, satisfies the orthogonality conditions
    $$ \int_0^\infty F_{\vec{n},\vec{m}}^{(II\mid \alpha,\beta,\nu)}(x) x^{k+\alpha} dx = 0,\quad k=0,\dots,m_1-1, $$
    $$ \int_0^\infty F_{\vec{n},\vec{m}}^{(II\mid \alpha,\beta,\nu)}(x) x^{k+\beta} dx = 0,\quad k=0,\dots,m_2-1. $$
We will again use the mixed type functions directly to generate the approximants. This is motivated by the fact that the system of weights $(e^{-x},E_{\nu+1}(x))$ forms a Nikishin system on $(0,\infty)$, since
    $$e^{x} E_{\nu+1}(x) = \frac{1}{\Gamma(\nu+1)} \int_{-\infty}^0 \frac{(-t)^\nu e^{t}}{x-t} dt, $$
see \cite[Eq. 3.383.10]{DLMF}. In that case, we can expect that $F_{\vec{n};1}^{(II\mid \alpha,\beta,\nu)}/F_{\vec{n};2}^{(II\mid \alpha,\beta,\nu)} \to e^{x} E_{\nu+1}(x)$ as $\sz{n}\to\infty$, e.g., along the diagonal, see \cite[Theorem 1.4]{Lop-Per}.
\medbreak

Explicit expressions for the Mellin transforms of the mixed type functions were determined in \cite[Section 4.1]{VAWolfs} for appropriate multi-indices. This leads to the integral representation below.
    
\begin{prop} \label{II_int}
Suppose that $\sz{n}=\sz{m}+1$ and $n_1+1\geq n_2$. Assume that $\alpha-\beta,\nu\not\in\Z$. Then,
   $$ F_{\vec{n},\vec{m}}^{(II\mid \alpha,\beta,\nu)}(x) = \int_{\mathcal{C}} \frac{(\alpha-s)_{m_1} (\beta-s)_{m_2}}{(s+\nu+1)_{n_2}} \frac{\Gamma(s+1)}{x^{s+1}} \frac{ds}{2\pi i}, $$
   in terms of a counterclockwise contour $\mathcal{C}$ that encloses $(-\infty,0]$ exactly once.
\end{prop}
\begin{proof}
It is known that the Mellin transform of $F_{\vec{n},\vec{m}}^{(II\mid \alpha,\beta,\nu)}(x) e^{-x}$ is given by
$$ \int_0^\infty F_{\vec{n},\vec{m}}^{(II\mid \alpha,\beta,\nu)}(x) e^{-x} x^{s-1} dx = \Gamma(s) \frac{(\alpha+1-s)_{m_1} (\beta+1-s)_{m_2}}{(s+\nu)_{n_2}},$$
see \cite[Lemma 4.1]{VAWolfs}. After taking the inverse Mellin transform, performing the change of variables $s\mapsto s+1$ and deforming the associated contour, we can obtain the stated result.
\end{proof}

Notice the duality between the integral representation in both settings (see Proposition \ref{I_int} and Proposition \ref{II_int}): the integrands are essentially each other's reciprocate after swapping the roles of $\vec{n}$ and ${\vec{m}}$.
\medbreak

In what follows, we will only consider the limiting case $\beta\to\alpha$ and take $\alpha=\nu=0$. For convenience, we will restrict to multi-indices $\vec{n} = (n,n+1)$ and $\vec{m} = (n,n)$. The approximants then arise as $F_{n}^{(II)}(x) = - e^{x} F_{(n,n+1),(n,n)}^{(II\mid 0,0,0)}(x)/n!$. 

\subsection{Quality} \label{GC_proof}

\textbf{Diophantine quality.} We will first obtain an explicit expression for the denominators of the approximants. 

\begin{prop} \label{II_form}
Denote $F_{n}^{(II)}(x) = - e^{x} F_{(n,n+1),(n,n)}^{(II\mid 0,0,0)}(x)/n!$. Then,
    $$F_{n}^{(II)}(x) = F_{n;1}^{(II)}(x) - F_{n;2}^{(II)}(x) e^{x} E_{1}(x),$$
where
$$\begin{aligned}
    F_{n;2}^{(II)}(x) &= \sum_{l=0}^{n} \binom{n}{l} \binom{n+l}{l}^2 \frac{x^l}{l!}.
\end{aligned}$$
\end{prop}
\begin{proof}
It was shown in \cite[Lemma 4.1]{VAWolfs} that the Mellin transform of $F_{n}^{(II)}(x) e^{-x}$ is given by
    $$ \int_0^\infty F_{n}^{(II)}(x) e^{-x} x^{s-1} dx  = - \frac{1}{n!} \Gamma(s) \frac{(1-s)_n^2}{(s)_{n+1}} . $$
We then use the definition of $F_{n}^{(II)}(x)$ to write
$$F_{n}^{(II)}(x) e^{-x} = \sum_{k=0}^{n-1} F_{n;1}^{(II)}[k] x^k e^{-x} -  \sum_{k=0}^{n} F_{n;2}^{(II)}[k] x^k E_1(x).$$ 
Note that, given a polynomial $p(x)\in\R[x]$, we have written $p[k]=p^{(k)}(0)/k!$ so that $p(x) = \sum_{k=0}^{\deg p} p[k] x^k$ (this notation will also appear in the rest of the paper). We have
    $$ \int_0^\infty x^{s+k-1} e^{-x} dx = \Gamma(s+k),\quad \int_0^\infty x^{s+k-1} E_1(x) dx = \int_1^\infty \frac{\Gamma(s+k)}{t^{s+k-1}} dt =   \frac{\Gamma(s+k)}{s+k}, $$
so that after dividing by $\Gamma(s)$, we obtain
\begin{equation} \label{II_MT}
    \sum_{k=0}^{n-1} F_{n;1}^{(II)}[k] (s)_k -  \sum_{k=0}^{n} F_{n;2}^{(II)}[k] \frac{(s)_k}{s+k} = - \frac{1}{n!} \frac{(1-s)_n^2}{(s)_{n+1}}.
\end{equation}
We can recover the coefficients $F_{n;2}^{(II)}[k]$ of $F_{n;2}^{(II)}(x)$ by comparing the poles at $s=-k$ for $k\in\{0,\dots,n\}$. Indeed,
    $$ \text{Res}_{s=-k}\left(\sum_{k=0}^{n-1} F_{n;1}^{(II)}[k] (s)_k -  \sum_{k=0}^{n} F_{n;2}^{(II)}[k] \frac{(s)_k}{s+k} \right) = \text{Res}_{s=-k}\left( - \frac{1}{n!} \frac{(1-s)_n^2}{(s)_{n+1}} \right) $$
and thus
$$ F_{n;2}^{(II)}[k] (-k)_k = \frac{1}{n!} \frac{(k+1)_n^2}{(-k)_{k}(1)_{n-k}}. $$
Elementary manipulations of the Pochhammer symbol, namely
    $$ (-k)_k = (-1)^k k!,\quad (k+1)_n = \binom{n+k}{k} n!,\quad (1)_{n-k}=(n-k)!, $$
then lead to the desired result.
\end{proof}

Observe that the denominators $F_{n;2}^{(II)}(1)=\sum_{l=0}^{n} \binom{n}{l} \binom{n+l}{l}^2 \frac{1}{l!}$ are similar to the denominators $\sum_{l=0}^{n} \binom{n}{l}^2 \binom{n+l}{l}^2$ used in Apéry's irrationality proof \cite{Apery} of $\zeta(3)$, but that a factor $n!/(n-l)!$ is missing from the $l$-th term.
\medbreak

We will now show that the approximants have the Diophantine properties below. In principle, it is also possible to use \eqref{II_MT} to find an explicit expression for the numerators $F_{n;1}^{(II)}(x)$ and to then study its Diophantine properties, however this would lead to a sub-optimal result.

\begin{prop}
    $ n! F_{n;1}^{(II)}(x), n! F_{n;2}^{(II)}(x) \in\Z[x]. $
\end{prop}
\begin{proof}
    The result about $F_{n;2}^{(II)}(x)$ follows easily from Proposition \ref{II_form}. In order to obtain the result for $F_{n;1}^{(II)}(x)$, we use the Euclidean division algorithm to find unique polynomials $q(s),r(s)\in \Z[s]$ of degree $n-1$ and $n$ respectively such that
        $$\frac{(1-s)_n^2}{(s)_{n+1}} = q(s) + \frac{r(s)}{(s)_{n+1}}. $$
    On the other hand, recall from \eqref{II_MT} that
    $$\frac{(1-s)_n^2}{(s)_{n+1}} = a(s) + \frac{b(s)}{(s)_{n+1}}, $$
    with 
    $$a(s) = \sum_{k=0}^{n-1} n! F_{n;1}^{(II)}[k] (s)_k,\quad b(s) = - \sum_{k=0}^{n} n! F_{n;2}^{(II)}[k] (s)_k \prod_{j=0,j\neq k}^{n} (s+j). $$
    Since $n! F_{n;2}^{(II)}[k]\in\Z$, Euclidean division again gives polynomials $b_q(s),b_r(s)\in \Z[s]$ of degree $n-1$ and $n$ respectively such that
    $$\frac{(1-s)_n^2}{(s)_{n+1}} = a(s) + b_q(s) + \frac{b_r(s)}{(s)_{n+1}}.$$
    We thus have $a(s) = q(s) - b_q(s)$ and therefore $a(s)\in\Z[s]$. The expansion $s^k=\sum_{l=0}^k c_{k,l} (s)_l$ has all $c_{k,l}\in\Z$, which then shows that $n! F_{n;1}^{(II)}[k]\in\Z$ as well.
\end{proof}  

\bigbreak

\textbf{Approximation quality.} We will use the same approach as with the approximants of $\gamma+\ln x$. We will start by generating a recurrence relation for the denominators $F_{n;2}^{(II)}(x)$ via the \textit{Zeilberger}-command in Maple (see \cite{A=B}) in order to study its asymptotics.

\begin{prop}
    The denominators $F_{n;2}^{(II)}(x)$ arise as solutions of a five-term recurrence relation $\sum_{k=0}^4 c_{n,k}(x) S_{n+k}(x)=0$ in which $c_{n,k}(x)\in\Z[x,n]$.
    For $x=1$, the coefficients are given by
    $$c_{n,0}(1) = -(n + 1)^3(n + 2)(729n^4 + 8190n^3 + 33657n^2 + 59656n + 38232), $$
    $$\begin{aligned}
        c_{n,1}(1) = (n + 2)&(2916n^7 + 51714n^6 + 385080n^5 + 1559443n^4 \\
        & + 3708086n^3 + 5178913n^2 + 3937800n + 1259352), \\
    \end{aligned} $$
    $$\begin{aligned}
        c_{n,2}(1) = &-4374n^8 - 43308n^7 - 17703n^6 + 1495014n^5 + 8678747n^4 \\
        & + 23155644n^3 + 33243300n^2 + 24757376n + 7476768, \\
    \end{aligned}$$
    $$\begin{aligned}
        c_{n,3}(1) = (n + 3)&(2916n^7 + 64107n^6 + 558186n^5 + 2519599n^4 \\
        & + 6371356n^3 + 8961012n^2 + 6375808n + 1702816),
    \end{aligned}$$
    $$c_{n,4}(1) = -(n + 3)(n + 4)^3(729n^4 + 5274n^3 + 13461n^2 + 13996n + 4772).$$
\end{prop}

It can again be shown that the three other linear independent solutions of the recurrence relation are $F_{n;1}^{(II)}(x)$ and the Stieltjes transforms of $F_n^{(II)}$ with respect to each type II orthogonality measure, i.e.
$$E_{n;1}^{(II)}(x)=\int_0^\infty \frac{F_n^{(II)}(t)}{x-t} dt ,\quad E_{n;2}^{(II)}(x)=\int_0^\infty \frac{F_n^{(II)}(t)}{x-t} \ln t \, dt.$$
\medbreak

The Birkhoff-Trjitzinsky theory now leads to the following result.

\begin{prop} \label{II_AQ_BT}
    There are four solutions $\{S_n^{(II)}(\omega)\mid \omega^4=x \}$ that form a basis over $\C$ for the space of solutions of the recurrence relation for $F_{n;2}^{(II)}(x)$ with asymptotic behavior
    $$S_{n}^{(I)}(\omega) = n^{-9/8} \exp\left(4\omega n^{3/4} + \frac{1}{2}\omega^2 n^{1/2} - \frac{3}{8}\omega^3 n^{1/4} \right) \left(1+\BO(n^{-1/4})\right),\quad n\to\infty.$$
\end{prop}
\begin{proof}
    This can be shown in the same way as Proposition \ref{I_AQ_BT}.
\end{proof}

The possible exponents for $x=1$ are summarized in the table below.
\begin{center}
\begin{tabular}{c||c|c|c}
    $S_n^{(II)}(\omega)$ & $n^{3/4}$ & $n^{1/2}$ & $n^{1/4}$ \\ \hline
    $1$ & $4$ & $1/2$ & $-3/8$ \\
    $i$ & $(4i)$ & $-1/2$ & $(3/8i)$ \\
    $-i$ & $(-4i)$ & $-1/2$ & $(-3/8i)$ \\
    $-1$ & $-4$ & $1/2$ & $3/8$ \\
\end{tabular}
\end{center}

Using the same strategy as before (by comparing with Laguerre polynomials in a negative argument, see Proposition \ref{I_AQ_denom}), we can associate a specific asymptotic behavior to the denominators.

\begin{prop}
Suppose that $x>0$. Then $F_{n,2}^{(II)}(x) \asymp S_{n}^{(I)}(\sqrt[4]{x})$ as $n\to\infty$.
\end{prop}

A variation of the saddle point method allows us to study the asymptotics of the errors via the integral representation in Proposition \ref{II_int}, similarly as in Proposition \ref{I_AQ_error}.

\begin{prop}
Suppose that $x>0$. Then,
    $$F_{n}^{(I)}(x) \in\text{span}\{\text{Re}(S_{n}^{(I)}(i\sqrt[4]{x})),\text{Im}(S_{n}^{(I)}(i\sqrt[4]{x})),S_n^{(I)}(-\sqrt[4]{x})\}.$$     
\end{prop}
\begin{proof}
It follows from Proposition \ref{II_int} that
$$F_n^{(II)}(x) = \frac{1}{n!} \int_{\mathcal{C}} g_n(s) x^{-s-1} \frac{ds}{2\pi i},\quad g_n(s) = \Gamma(s+1) \frac{(-s)_n^2}{(s+1)_{n+1}}.$$
The integrand $g_n(s)$ is essentially the reciprocate of the integrand $f_n(s)$ that is used in the contour integral representation for $F_{n}^{(I)}(x)$ (see Proposition \ref{I_int}). Indeed, we have
    $$g_n(s) =  \frac{1}{f_n(s)} \frac{1}{(s+n+1)(-s+n)^2} , $$
and thus
    $$  \frac{n^{3/4}}{n!} g_n(n^{3/4}t) x^{-n^{3/4}t-1} \sim - x^{-1} t^2 n^{-3/4} e^{-n^{3/4}\p_n(t)},  $$
with
    $$\p_n(t) = t(1-4\ln t+\ln x) + (t+n^{1/4})\ln(1+n^{-1/4}t) + 2 (t-n^{1/4}) \ln(1-n^{-1/4}t).$$  

By applying the (modified) saddle point method, we can show that $F_n^{(II)}(x)$ is asymptotic to an appropriate combination of $n^{-9/8}\exp\left(-n^{3/4}\p_n(s^\ast_n(\omega))\right)$ for $\omega^4=x$. Note that these possible asymptotics are exactly those obtained in Proposition \ref{II_AQ_BT} via the Birkhoff-Trjitzinsky theory. The specific asymptotics that appear in the combination here depend on the approximate saddle points that the curve of steepest descent, which the contour $n^{3/4}\mathcal{C}$ needs to be deformed to, passes through. In this setting, the contour can't be deformed to a curve that passes through the approximate saddle point $s_n^\ast(-x^{1/4})$ on the negative real line because then it would have to cross the strip $[-n^{1/4},0]$ inside $n^{-3/4}\mathcal{C}$ where the integrand has its poles.
\end{proof}

\section{Connection}

From Proposition \ref{I_AQ_BT} and Proposition \ref{II_AQ_BT}, one can observe that the asymptotics of the solutions of the recurrence relations in both settings are closely related, see, e.g., the tables with the potential powers for $x=1$ below. 

\begin{center}
\begin{tabular}{c||c|c|c}
    $S_n^{(I)}(\omega)$ & $n^{3/4}$ & $n^{1/2}$ & $n^{1/4}$ \\ \hline
    $1$ & $4$ & $-1/2$ & $-3/8$ \\
    $i$ & $(4i)$ & $1/2$ & $(3/8i)$ \\
    $-i$ & $(-4i)$ & $1/2$ & $(-3/8i)$ \\
    $-1$ & $-4$ & $-1/2$ & $3/8$ \\
\end{tabular}
\qquad
\begin{tabular}{c||c|c|c}
    $S_n^{(II)}(\omega)$ & $n^{3/4}$ & $n^{1/2}$ & $n^{1/4}$ \\ \hline
    $1$ & $4$ & $1/2$ & $-3/8$ \\
    $i$ & $(4i)$ & $-1/2$ & $(3/8i)$ \\
    $-i$ & $(-4i)$ & $-1/2$ & $(-3/8i)$ \\
    $-1$ & $-4$ & $1/2$ & $3/8$ \\
\end{tabular}
\end{center}

In the first place, this can be explained by the fact that the integrand in the integral representation for the associated mixed type functions in Proposition \ref{I_int} and Proposition \ref{II_int} are essentially each other's reciprocal, hence the powers should be each other's opposite. A more deeper connection can be established through relations similar as in \eqref{L_I_cof} that arise via the underlying Riemann-Hilbert problems: the rows in the table on the left arise as sums of three different rows in the table on the right (and vice versa). This connection is stronger in the sense that it actually enables us to study the asymptotics of objects in one setting by studying the asymptotics of specific objects in the other. As such, we can also give an alternative proof of Theorem \ref{Res_EC} solely by analyzing the asymptotics of the objects associated with system $(I)$. In what follows, we will provide the key ideas on how this can be done (a complete proof is more cumbersome and will be omitted).
\medbreak

The Riemann-Hilbert problem for mixed type multiple orthogonal polynomials, see \cite[Lemma 3.3]{DaemsKuijl}, gives the following relation
    $$Y_{\vec{n},\vec{m}}^{(I)} = \begin{pmatrix}
        0 & -I_2 \\
        I_2 & 0
    \end{pmatrix}
    (Y_{\vec{m},\vec{n}}^{(II)})^{-T}
    \begin{pmatrix}
        0 & I_2 \\
        -I_2 & 0
    \end{pmatrix},
    $$
if $|\vec{n}|=|\vec{m}|$, where, up to a particular scaling of each row,
    $$ Y_{\vec{n},\vec{m}}^{(\ast)} = \begin{pmatrix} 
    F_{\vec{n}+\vec{e}_1,\vec{m};1}^{(\ast)} & F_{\vec{n}+\vec{e}_1,\vec{m};2}^{(\ast)} & E_{\vec{n}+\vec{e}_1,\vec{m};1}^{(\ast)} & E_{\vec{n}+\vec{e}_1,\vec{m};2}^{(\ast)} \\ 
    F_{\vec{n}+\vec{e}_2,\vec{m};1}^{(\ast)} & F_{\vec{n}+\vec{e}_2,\vec{m};2}^{(\ast)} & E_{\vec{n}+\vec{e}_2,\vec{m};1}^{(\ast)} & E_{\vec{n}+\vec{e}_2,\vec{m};2}^{(\ast)} \\ 
    F_{\vec{n},\vec{m}-\vec{e}_1;1}^{(\ast)} & F_{\vec{n},\vec{m}-\vec{e}_1;2}^{(\ast)} & E_{\vec{n},\vec{m}-\vec{e}_1;1}^{(\ast)} & E_{\vec{n},\vec{m}-\vec{e}_1;2}^{(\ast)} \\ 
    F_{\vec{n},\vec{m}-\vec{e}_2;1}^{(\ast)} & F_{\vec{n},\vec{m}-\vec{e}_2;2}^{(\ast)} & E_{\vec{n},\vec{m}-\vec{e}_2;1}^{(\ast)} & E_{\vec{n},\vec{m}-\vec{e}_2;2}^{(\ast)} \end{pmatrix}. $$
Rows associated with multi-indices $(\vec{n}+\vec{e}_k,\vec{m})$ are scaled such that the entry $F_{\vec{n}+\vec{e}_k,\vec{m};k}^{(\ast)}(x)$ is monic. Rows associated with multi-indices $(\vec{n},\vec{m}-\vec{e}_k)$ are scaled such that mixed type function $F_{\vec{n},\vec{m}-\vec{e}_k}^{(\ast)}(x)$ satisfies the type I normalization condition with respect to the $k$-th type II orthogonality weight in system $(\ast)$, i.e.
$$ \int_0^\infty F_{\vec{n},\vec{m}-\vec{e}_k}^{(\ast)}(x) x^{m_k-1} w_k^{(\ast)}(x) dx = 1. $$
Since this particular scaling ensures that $\det(Y_{\vec{m},\vec{n}}^{(II)})=1$, it follows from the adjugate formula that the inverse transpose of $Y_{\vec{m},\vec{n}}^{(II)}$ is equal to its cofactor matrix $\text{CF}(Y_{\vec{m},\vec{n}}^{(II)})$. We therefore have, up to an appropriate scalar,

$$\text{CF}(Y_{\vec{m},\vec{n}}^{(II)}) = \begin{vmatrix} 
    E_{\vec{n},\vec{m}-\vec{e}_1;1}^{(I)} & E_{\vec{n},\vec{m}-\vec{e}_1;2}^{(I)} & 
    -F_{\vec{n},\vec{m}-\vec{e}_1;1}^{(I)} & -F_{\vec{n},\vec{m}-\vec{e}_1;2}^{(I)} \\ 
    E_{\vec{n},\vec{m}-\vec{e}_2;1}^{(I)} & E_{\vec{n},\vec{m}-\vec{e}_2;2}^{(I)} & 
    -F_{\vec{n},\vec{m}-\vec{e}_2;1}^{(I)} & -F_{\vec{n},\vec{m}-\vec{e}_2;2}^{(I)} \\ 
    -E_{\vec{n}+\vec{e}_1,\vec{m};1}^{(I)} & -E_{\vec{n}+\vec{e}_1,\vec{m};2}^{(I)} & F_{\vec{n}+\vec{e}_1,\vec{m};1}^{(I)} & F_{\vec{n}+\vec{e}_1,\vec{m};2}^{(I)} \\ 
    -E_{\vec{n}+\vec{e}_2,\vec{m};1}^{(I)} & -E_{\vec{n}+\vec{e}_2,\vec{m};2}^{(I)} & F_{\vec{n}+\vec{e}_2,\vec{m};1}^{(I)} & F_{\vec{n}+\vec{e}_2,\vec{m};2}^{(I)} \end{vmatrix}.$$

Hence, up to an appropriate scalar, the mixed type functions in system $(I)$ can be written as  
\begin{equation} \label{I_RH}
    F_{\vec{n}+\vec{e}_2,\vec{m}}^{(I)} = \begin{vmatrix} F_{\vec{m}+\vec{e}_1,\vec{n};1}^{(II)} & F_{\vec{m}+\vec{e}_1,\vec{n};2}^{(II)} & E_{\vec{m}+\vec{e}_1,\vec{n}}^{(II)} \\ F_{\vec{m}+\vec{e}_2,\vec{n};1}^{(II)} & F_{\vec{m}+\vec{e}_2,\vec{n};2}^{(II)} & E_{\vec{m}+\vec{e}_2,\vec{n}}^{(II)} \\ 
    F_{\vec{m},\vec{n}-\vec{e}_1;1}^{(II)} & F_{\vec{m},\vec{n}-\vec{e}_1;2}^{(II)} & E_{\vec{m},\vec{n}-\vec{e}_1}^{(II)} \end{vmatrix},
\end{equation}
in terms of a particular combination $ E_{\vec{m},\vec{n}}^{(II)}(z) = E_{\vec{m},\vec{n};1}^{(II)}(z) \ln z - E_{\vec{m},\vec{n};2}^{(II)}(z)$ of the two errors in the associated Hermite-Padé approximation problem
    $$ E_{\vec{m},\vec{n};1}^{(II)}(z) = \int_0^\infty \frac{F_{\vec{m},\vec{n};1}^{(II)}(x)}{z-x} dx,\quad E_{\vec{m},\vec{n};1}^{(II)}(z) = \int_0^\infty \frac{F_{\vec{m},\vec{n};1}^{(II)}(x)}{z-x} \ln x \, dx. $$
In this context, the appropriate scalar is given by
\begin{equation} \label{I_RH_scalar}
    \mathcal{F}_{\vec{n}+\vec{e}_2,\vec{m}} = \frac{\text{LC}(F_{\vec{n}+\vec{e}_2,\vec{m};2}^{(I)})}{\text{LC}(F_{\vec{m}+\vec{e}_1,\vec{n};1}^{(II)}) \cdot \text{LC}(F_{\vec{m}+\vec{e}_2,\vec{n};2}^{(II)}) \cdot \int_0^\infty F_{\vec{m},\vec{n}-\vec{e}_1;1}^{(II)}(x) x^{n_1-1} dx }.
\end{equation}
Taking $\vec{n}=(n+1,n)$ and $\vec{m}=(n+1,n)$ in \eqref{I_RH} then allows us to express the errors $F_{n}^{(I)}$ and denominators $F_{n;2}^{(I)}$ in system $(I)$ in terms of objects from the dual system $(II)$. After scaling the objects in both systems as in Proposition \ref{I_form} and Proposition \ref{II_form}, it can be expected that the scalar in \eqref{I_RH_scalar} is asymptotically negligible (compared to sub-exponential growth). Indeed, the objects that appear in the formula are then normalized to grow sub-exponentially and a determinant with entries that grow at most sub-exponentially grows at most sub-exponentially itself. In order to actually prove this, we require knowledge about the mixed type functions for more general multi-indices than in Proposition \ref{I_form} and Proposition \ref{II_form} (but that can be obtained using the same ideas).
\medbreak

It then remains to describe the asymptotics of the modified error $E_{\vec{m},\vec{n}}^{(II)}(z)$. This can be done by analyzing the following integral representation
$$E_{\vec{m},\vec{n}}^{(II)}(z) = (-1)^{n_1+n_2} \int_0^\infty\int_0^\infty\int_0^\infty  \frac{x^{n_2} y^{n_2} t^{m_2} e^{-x(1+t)} }{(z+xy)^{n_2+1}(1+y)^{n_1+1}(1+t)^{m_2+1}} dx dy dt,$$
which further generalizes the integral Rivoal considered in \cite[Lemma 1]{Riv}. Doing so allows us to prove that, e.g., $E_n^{(II)}(z) \asymp S_n(-\sqrt[4]{z})$ as $n\to\infty$ for $z>0$.

\section{Further improvement}

In what follows, we will propose a construction that could lead to a further improvement of the approximants of Euler's constant in Theorem \ref{Res_EC}. Since the resulting quality is still of a similar kind as before, we will only provide the key ideas.
\medbreak

In order to obtain approximants for which the absolute error has a faster sub-exponential decay than before, we can apply the the differential operator 
    $$F(x) \mapsto \frac{1}{n!}\frac{d^n}{dx^n} [x^n F(x)]$$
$p-1$ more times to the linear forms $F_n^{(I)}(x)$ obtained in Proposition \ref{I_form}. Doing so yields a function of the form
$$F_n^{(I\mid p)} = n!^{2-p} \int_{\Sigma} \frac{(t+1)_{n}^p}{(-t)_{n+1}^2}  \frac{x^t}{\Gamma(t+1)} \frac{dt}{2\pi i},$$
where $\Sigma$ is a counterclockwise contour in $\{t\in\C \mid \text{Re}(t)>-1\}$ that encloses $(0,\infty)$ exactly once. We still have
$$F_n^{(I\mid p)}(x) = F_{n;1}^{(I\mid p)}(x) - F_{n;2}^{(I\mid p)}(x) (\gamma+\ln x),$$
but now the denominators are of the form
    $$F_{n;2}^{(I\mid p)}(x) = \sum_{k=0}^n \binom{n}{k}^2 \binom{n+k}{k}^p \frac{x^k}{k!}. $$
It can be shown that the Diophantine properties in Proposition \ref{I_DQ} remain unchanged. By analyzing the asymptotic properties of the contour integral as in Proposition \ref{I_AQ_error}, one should be able to prove that the main contributions to the possible asymptotics for $F_n^{(I\mid p)}(x)$ are determined by $\exp((p+3)\omega n^{(p+2)/(p+3)}(1+o(1)))$ in terms of $\omega\in\C$ with $\omega^{p+3}=x$. Similarly as before, due to the structure of the contour, it then follows that, when $x>0$, the solution $\omega=x^{1/(p+3)}$ on the positive real line can't contribute to the asymptotics for $F_n^{(I\mid p)}(x)$. This solution should again correspond to the asymptotics of the denominators $F_{n;2}^{(I\mid p)}(x)$. In that case, the absolute error indeed has a faster sub-exponential decay
    $$ \left| (\gamma + \ln x) - \frac{F_{n}^{(I\mid p)}(x)}{F_{n;2}^{(I\mid p)}(x)}\right| = \BO(\exp(-(p+3) x^{1/(p+3)}n^{(p+2)/(p+3)}(1+o(1)))),\quad n\to\infty, $$
and the approximants of $\gamma$ are of an overall better quality in the sense of Lemma \ref{OQ}:
    $$ r^{(I\mid p)}(n) = \frac{p+3}{n^{1/(p+3)}\ln n} (1+o(1)),\quad n\to\infty . $$
It was brought to our attention that a construction in \cite[Theorem 1.1]{PilehroodEC} also leads to approximants of $\gamma$ with faster sub-exponential decay. There one showed that for
    $$Q_n^{\text{Pil},a} = \sum_{k=0}^n \binom{n}{k}^a k! ,\quad P_n^{\text{Pil},a} = \sum_{k=0}^n \binom{n}{k}^a k! (aH_{n-k}-(a-1) H_k),$$
one has
$$ \left| \gamma - \frac{P_n^{\text{Pil},a}}{Q_n^{\text{Pil},a}} \right| = \BO(\exp(-a (1-\cos(2\pi/a)) n^{(a-1)/a} (1+o(1)))),$$ 
$$Q_n^{\text{Pil},a} \asymp n! n^{-(a/2+1/(2a))} \exp( a n^{(a-1)/a} (1+o(1))), $$
as $n\to\infty$. The overall quality is then determined by
$$ r_{\text{Pil},a}(n) = \frac{a(1-\cos(2\pi/a))}{n^{1/a}\ln n} (1+o(1)),\quad n\to\infty, $$
which has a similar dependence on $n$ as in the above construction, but with a constant $a(1-\cos(2\pi/a))$ that tends to $0$ as $a\to\infty$.

\end{document}